\documentclass{amsart}
\makeindex
\usepackage[all, cmtip]{xy}
\usepackage{amssymb}
\usepackage{amsmath}
\usepackage{amsthm}
\usepackage{amscd}
\usepackage{amsfonts}
\usepackage{multirow}
\usepackage[usenames,dvipsnames,svgnames,table]{xcolor}
\usepackage{epic}
\usepackage{pgf,tikz}
\usepackage{enumitem}

\usepackage[bookmarks=true,bookmarksnumbered=true,
 pdftitle={},
 pdfsubject={},
 pdfauthor={},
 pdfkeywords={TeX; dvipdfmx; hyperref; color;},
 colorlinks=true, linkcolor=Emerald, citecolor=ForestGreen]
 {hyperref}

\usetikzlibrary{arrows}

\newtheorem{theorem}{Theorem}[section]
\newtheorem{lemma}[theorem]{Lemma}
\newtheorem{proposition}[theorem]{Proposition}

\theoremstyle{corollary}
\newtheorem{corollary}[theorem]{Corollary}
\newtheorem{problem}[theorem]{Problem}
\theoremstyle{definition} 
\newtheorem{definition}[theorem]{Definition}
\newtheorem{definition-lemma}[theorem]{Definition-Lemma}

\newtheorem{example}[theorem]{Example}

\newtheorem{remark}[theorem]{Remark}
\theoremstyle{remark}

\numberwithin{equation}{section}

\newcommand{\R}{\mathbb{R}}
\newcommand{\Z}{\mathbb{Z}}
\newcommand{\N}{\mathbb{N}}
\newcommand{\Q}{\mathbb{Q}}
\newcommand{\mc}{\mathcal}
\def\P{\mathbb{P}}

\def\B{\operatorname{\bf{B}}}  

\def\Pic{\operatorname{Pic}}

\def\mult{\operatorname{mult}}

\def\sigmanum{\sigma_{\text{num}}}
\def\Cent{\operatorname{Cent}}
\def\Supp{\operatorname{Supp}}

\def\oa{\overline{a}}

\def\nnef{\operatorname{Nnef}}

\def\nklt{\operatorname{Nklt}}
\def\pnklt{\operatorname{pNklt}}

\def\epnklt{\epsilon\operatorname{-pNklt}}
\def\spnklt{\operatorname{spNklt}}
\def\espnklt{\epsilon\operatorname{-spNklt}}

\title[Potentially non-klt locus and its applications]{Potentially non-klt locus and its applications}

\begin{document}

\author{Sung Rak Choi}
\address{Center for Geometry and Physics, Institute for Basic Science (IBS), Pohang 37673, Republic of Korea}
\curraddr{Department of Mathematics, Yonsei University, Seoul 03722, Republic of Korea}
\email{sungrakc@yonsei.ac.kr}

\author{Jinhyung Park}
\address{School of Mathematics, Korea Institute for Advanced Study (KIAS), Seoul 02455, Republic of Korea}
\email{parkjh13@kias.re.kr}


\date{\today}

\begin{abstract}
We introduce the notion of potentially klt pairs for normal projective varieties with pseudoeffective anticanonical divisor.
The potentially non-klt locus is a subset of $X$ which is birationally transformed precisely into the non-klt locus on a $-K_X$-minimal model of $X$.
We prove basic properties of potentially non-klt locus in comparison with those of classical non-klt locus.
As applications, we give a new characterization of varieties of Fano type,
and we also improve results on the rational connectedness of uniruled varieties with pseudoeffective anticanonical divisor.
\end{abstract}

\maketitle
\tableofcontents

\section{Introduction}
Fano varieties appear naturally in various areas such as algebraic geometry, representation theory, and theoretical physics. It is well known that varieties of Fano type share many of the properties of Fano varieties. Furthermore, they play very important roles in birational geometry, and the characterization problem has attracted considerable attentions in a variety of flavors (see e.g., \cite{SS}, \cite{B}, \cite{GOST}, \cite{HP2}, \cite{CG}, \cite{CHP}).
Note that a smooth projective surface $S$ is of Fano type if and only if $-K_S$ is big and $(S, N)$ is a klt pair where $-K_S = P+N$ is the Zariski decomposition. One may wonder whether the analogous characterization of Fano type varieties via Zariski decompositions exists in higher dimensions.
An obvious obstruction is the fact that the Zariski decompositions do not exist in general
(see e.g., \cite{Cu}, \cite{Na}, \cite{John}).
Among the various generalizations of Zariski decompositions, the \emph{divisorial Zariski decompositions} have the feature that they always exist.
It is tempting to expect for a variety $X$ to be of Fano type if $-K_X$ is big and $(X,N)$ is klt where $N$ is the negative part of the divisorial Zariski decomposition of $-K_X$.
However, it is false in general even if $X$ is smooth, $N=0$, and we can run the $-K_X$-minimal model program.
Any end product of this program has nef and big anticanonical divisor, but the resulting model can contain singularities worse than klt or even lc singularities  (see \cite[Example 5.1]{CG} and Example \ref{semifano example}).
One can check that such a variety $X$ is not of Fano type.
We note that the $-K_X$-minimal model program modifies the locus where $-K_X$ is not nef, but if such locus is deeply embedded in $X$, then the modification creates bad singularities.
To remedy this situation we define the potentially klt pairs as a variant of klt pairs.

In this paper, we work over an algebraically closed field of characteristic zero.

\begin{definition}
Let $(X, \Delta)$ be a pair such that $-(K_X+\Delta)$ is a pseudoeffective $\R$-Cartier divisor. For a log resolution $f : Y \to X$ of $(X,\Delta)$ and for a prime divisor $E$ on $Y$, we define the \emph{potential discrepancy} of $(X,\Delta)$ at $E$ as
$$
\oa(E; X,\Delta):=a(E;X,\Delta)-\mult_E N
$$
where $a(E;X,\Delta)$ denotes the usual discrepancy and $-f^*(K_X + \Delta)=P+N$ is the divisorial Zariski decomposition.
We call $(X, \Delta)$  a \emph{potentially klt} (resp. \emph{potentially lc}) pair if
$$
\mathfrak{A}(X, \Delta):=\inf_{f,E} \{ \oa(E; X, \Delta) \} >-1 ~~(\text{resp.} \geq -1)
$$
where $\inf$ is taken over the log resolutions $f:Y\to X$ of $(X, \Delta)$ and prime divisors $E$ on $Y$.
The \emph{potentially non-klt locus} $\pnklt(X,\Delta)$ is defined as
$$
\pnklt(X, \Delta) := \bigcap_{\epsilon>0} \espnklt(X, \Delta)
$$
where $\espnklt(X, \Delta) :=  \bigcup_{f, E}\{ \Cent_X E\;|\;\oa(E; X, \Delta)  \leq -1+\epsilon\}$.
\end{definition}

Note that the potential discrepancies $\oa(E; X,\Delta)$
not only measure the singularities of the pair $(X,\Delta)$ but also capture the positivity of $-(K_X+\Delta)$.
Thus their behavior is more subtle in nature than the usual singularities of pairs.
However, it turns out that the potentially klt, lc pairs share many similar properties
of the usual klt, lc pairs (see Remarks \ref{rmk-subtle1} and \ref{rmk-subtle2}).
Furthermore, unlike the typical singularities which reflect the local properties,
our definition of potential pairs encode the global property of the variety as well.

One of the basic properties of potentially non-klt locus $\pnklt(X,\Delta)$
(which is actually the motivation of the definition) is
that it is precisely the locus on $X$ that becomes the non-klt locus on a $-(K_X+\Delta)$-minimal model of $(X,\Delta)$ (see Corollary \ref{cor-MMP on pnklt}).
For this reason, they are said to be \emph{potential}.
We refer to Sections \ref{sec-pklt} and \ref{sec-pnklr locus} for more details on
the potential pairs.


As the first application of potentially klt pairs, we give a new characterization of Fano type varieties
using the divisorial Zariski decompositions.

\begin{theorem}[=Corollary \ref{cor-FT=pklt(X,N)}]\label{main thrm1}
A smooth projective variety $X$ is of Fano type if and only if $-K_X$ is big and $(X, N)$ is potentially klt where $-K_X=P+N$ is the divisorial Zariski decomposition.
\end{theorem}

Note that a Fano type variety is nothing but a klt Calabi-Yau type variety with big anticanonical divisor.
It is natural to ask whether Theorem \ref{main thrm1} can be generalized to varieties of klt Calabi-Yau type.
Unfortunately, there exists a potentially klt surface $S$ which is not of Calabi-Yau type (see Example \ref{nikex}).
Despite this inconvenience, we can still obtain an analogous statement to Theorem \ref{main thrm1}
using a variation of potential pairs which is defined with the $s$-decomposition (see Theorem \ref{thrm-charcy}).

We next apply the theory of potentially klt pairs to the rational connectedness of uniruled varieties with pseudoeffective anticanonical divisor.

\begin{theorem}[=Theorems \ref{thrm-RC via pnklt} and \ref{thrm-RC pseff}]\label{main thrm2}
Let $X$ be a normal projective variety such that $-(K_X+\Delta)$ is $\R$-Cartier for some effective $\R$-divisor $\Delta$.
\begin{enumerate}
 \item[$(1)$] If $-(K_X + \Delta)$ is big, then $X$ is rationally  connected modulo $\pnklt(X, \Delta)$.
 \item[$(2)$] If $-(K_X + \Delta)$ is pseudoeffective and $\Delta \neq 0$, then $X$ is rationally  connected modulo either $\pnklt(X, \Delta)$ or every irreducible component of $\Delta$.
\end{enumerate}
\end{theorem}

We prove that $\pnklt(X, \Delta)$ or every irreducible component of $\Delta$ in the case (2) dominates $Z$ via the maximally rationally connected fibration $\pi : X \dashrightarrow Z$ as in \cite{Z} and \cite{BP}.
Note that every ruled surface $S$ with the canonical ruling $\pi : S \to C$ has a section. Then $S$ is rationally connected modulo this section. Even though we do not know the existence of such section for the maximally rationally connected fibration, we can still pick a subvariety dominating the maximally rationally connected quotient.
We need the assumption $\Delta \neq 0$ in Theorem \ref{main thrm2} (2) for $X$ to be uniruled.
Furthermore, there exists a variety $X$ which is rationally connected modulo every irreducible component $\Delta$
but not modulo $\pnklt(X, \Delta)$ (Example \ref{secondcaseoccurs}).

It was first shown in \cite{KMM} and \cite{C} that smooth Fano varieties are rationally connected, and this result was generalized to Fano type varieties in \cite{Z} and \cite{HM2}.
The paper \cite{BP} further extends these results by showing that $X$ is rationally connected modulo $\nklt(X,\Delta)\cup\nnef(-(K_X+\Delta))$ when $-(K_X+\Delta)$ is big.
Theorem \ref{main thrm2} improves these results since by Lemma \ref{lem-nklt+nnef=pnklt}, we have
$$
\pnklt(X,\Delta) \subseteq \nklt(X,\Delta)\cup\nnef(-(K_X+\Delta))
$$
and the inclusion is strict in general (Example \ref{ex-pnklt,nklt}).
Furthermore, our result is strict in the sense that
there actually exists a variety $X$ which is only rationally connected modulo $\nklt(X,\Delta)\cup\nnef(-(K_X+\Delta))$
but not modulo any component outside $\pnklt(X,\Delta)$ (see Example \ref{improvebp}).

The main ingredient in the proofs of our main theorems is the following fundamental property of $\pnklt(X,\Delta)$.

\begin{proposition}[=Propositions \ref{prop-pnkltR} and \ref{pnklt-locus-pseff}]\label{main ingredient}
Let $X$ be a normal projective uniruled variety such that $-(K_X+\Delta)$ is $\R$-Cartier for some effective $\R$-divisor $\Delta$.
\begin{enumerate}
 \item[$(1)$] If $-(K_X+\Delta)$ is big, then there is an effective $\R$-Cartier divisor $D$ such that $D \sim_{\R} -(K_X + \Delta)$ and $\nklt(X, \Delta+D)=\pnklt(X, \Delta)$.
 \item[$(2)$] If $-(K_X+\Delta)$ is pseudoeffective, then there is an effective $\Q$-Cartier divisor $D$ such that $D \sim_{\Q} -(K_X + \Delta) + A$ for some ample $\R$-divisor $A$ and $\nklt(X, \Delta+D)\subseteq\pnklt(X, \Delta)$.
\end{enumerate}
\end{proposition}

In particular, in the case (1), $\pnklt(X, \Delta)$ is a connected Zariski closed subset of $X$
(Corollary \ref{cor-pnklt big closed} and Proposition \ref{connedted-pnklt}).

This paper is organized as follows.
In Section \ref{sec-prelim}, we recall the basic notions such as singularities of pairs,
Zariski decomposition, and rational connectedness.
We define potentially klt, lc pairs and prove some of their basic properties in Section \ref{sec-pklt}.
Section \ref{sec-pnklr locus} is devoted to show the fundamental properties of potentially non-klt locus including Proposition \ref{main ingredient}.
Theorems \ref{main thrm1} and \ref{main thrm2} are proved in Sections \ref{sec-characteirze FT} and \ref{sec-RCC of pklt}, respectively.

\section{Preliminaries}\label{sec-prelim}
In this section, we collect basic notions and facts that will be used throughout the paper.
We follow the standard definitions from \cite{KM}, \cite{Na}, and \cite{K1}.

\subsection{Pairs}
A \emph{pair} $(X, \Delta)$ consists of a normal projective variety $X$ and an effective $\R$-divisor $\Delta$ on $X$ such that $K_X + \Delta$ is $\R$-Cartier.
For a log resolution $f : Y \rightarrow X$ of $(X, \Delta)$, we have
$$
K_Y = f^{*}(K_X + \Delta) + \sum a_i E_i
$$
where each $E_i$ is a prime divisor on $Y$ and
$a_i = a(E_i; X, \Delta)$ is the \emph{discrepancy} of $(X,\Delta)$ at $E_i$.
A pair $(X, \Delta)$ is called \emph{Kawamata log terminal} (\emph{klt} for short) (resp. \emph{log canonical} (\emph{lc} for short)) if every $a_i > -1$ (resp. $a_i \geq -1$).

A pair $(X, \Delta)$ is called a \emph{Fano pair} (resp. \emph{weak Fano pair}) if $-(K_X + \Delta)$ is ample (resp. nef and big).
A normal projective variety $X$ is said to be \emph{of Fano type} if there exists an effective divisor $\Delta$ on $X$
such that $(X, \Delta)$ is a klt Fano pair.
A pair $(X, \Delta)$ with an effective $\Q$-divisor $\Delta$ is called a \emph{Calabi-Yau pair} if $-(K_X + \Delta) \sim_{\Q} 0$.
A normal projective variety $X$ is said to be \emph{of Calabi-Yau type} (resp. \emph{of klt Calabi-Yau type})
if there is an effective $\Q$-divisor $\Delta$ on $X$ such that $(X, \Delta)$ is an lc Calabi-Yau pair (resp. klt Calabi-Yau pair).
It is easy to see that $X$ is a variety of Fano type if and only if $X$ is a variety of klt Calabi-Yau type and $-K_X$ is big.

\subsection{Non-nef locus}
Let $D$ be a pseudoeffective $\R$-divisor on a normal projective variety $X$ and $\sigma$ a divisorial valuation of $X$.
If $D$ is big, then we define \emph{the asymptotic valuation} of $\sigma$ at $D$ as
$$
\sigmanum(D):=\inf\{\sigma(D')\mid D\equiv D'\geq 0\}.
$$
If $D$ is only pseudoeffective, then we define $\sigmanum(D):=\lim_{\epsilon\to 0+}\sigmanum(D+\epsilon A)$
for some ample divisor $A$. (This definition is independent of the choice of $A$.)
The number $\sigmanum(D)$ depends only on the numerical class $[D]\in\text{N}^1(X)_\R$.
It is a birational invariant: for a birational morphism $f:Y\to X$, we have $\sigmanum(D)=\sigmanum(f^*(D))$ (see \cite{BBP}). We define the \emph{non-nef locus} of $D$ as
$$
\nnef(D):=\bigcup\Cent_X\sigma
$$
where the union is taken over all divisorial valuations $\sigma$ such that $\sigmanum(D)>0$.
We define $\nnef(D)=X$ is $D$ is not pseudoeffective.

The stable base locus of a $\Q$-divisor $D$ on a normal projective variety $X$ is denoted by $\B(D)$.
It is known that $\B(D)$ coincides with the base locus $b(|mD|)$ of the linear system $|mD|$ for
all sufficiently large and divisible integer $m>0$. Note that
$\nnef(D) \subseteq \B(D)$.

\subsection{Zariski decomposition}
Let $X$ be a $\Q$-factorial  normal projective variety and $D$ a pseudoeffective $\R$-divisor on $X$.
We define the \emph{negative part} $N=N(D)$ of $D$ as
$$
N(D):=\sum_{\sigma} \sigmanum (D)E_\sigma
$$
where the summation runs over all divisorial valuations $\sigma$ of $X$ such that
the centers $\Cent_X\sigma=E_\sigma$ are prime divisors on $X$.
It is known that there are only finitely many divisorial valuations $\sigma$
with divisorial centers $E_\sigma$ such that $\sigmanum(D)>0$.
The following decomposition is called the \emph{divisorial Zariski decomposition} of $D$:
$$
D= P(D)+N(D)
$$
where $P=P(D):=D-N$ is called the \emph{positive part} of $D$.

\begin{lemma}[{\cite[Chapter III]{Na}}]\label{lem-divzd}
Let $X$ be a $\Q$-factorial normal projective variety and $D$ a pseudoeffective $\R$-divisor with the divisorial Zariski decomposition $D=P+N$.
\begin{enumerate}
\item[$(1)$] $P$ is movable, i.e., $\nnef(P)$ has no divisorial components.
\item[$(2)$] $P$ is maximal in the sense that if $L$ is movable and $L \leq D$, then $L \leq P$.
\item[$(3)$] Let $E_1, \ldots, E_k$ be mutually distinct prime divisors. If $D$ is big, then for any $\epsilon >0$, there is an effective $\R$-divisor $P_0 \sim_{\R} P$ such that $\mult_{E_i} P_0 < \epsilon$ for any $i$.
\end{enumerate}
\end{lemma}

We now briefly recall the $s$-decomposition.
For more details, we refer to \cite{S} and \cite{P}.
Let $X$ be a $\Q$-factorial normal projective variety and $D$ be an effective $\Z$-divisor on $X$. By \cite[Lemma 1.10]{P}, there is a unique effective $\Z$-divisor $M$ such that (1) $M \leq D$, (2) $H^0(\mathcal{O}_X(M))=H^0(\mathcal{O}_X(D))$, and (3) if for a divisor $L \leq D$ with $H^0(\mathcal{O}_X(L))=H^0(\mathcal{O}_X(D))$, then $L \geq M$. We denote by $M_{\Z}(D)$ for such a divisor.

Now consider an effective $\Q$-divisor $D$ on $X$. We define $M(D):=M_{\Z}(\lfloor D \rfloor)$. If we put
$$
P_s=P_s(D):= \limsup_{n \to \infty} \frac{M(nD)}{n},
$$
then $nD \geq nP_s(D) \geq M(nD)$ for any integer $n>0$. The following decomposition is called the \emph{$s$-decomposition} of $D$:
$$
D=P_s(D)+N_s(D)
$$
where $N_s=N_s(D):=D-P_s(D)$. We can alternatively define as
$$
N_s=N_s(D):=\inf\{L \mid L \sim_{\Q} D, L \geq 0\}.
$$

We define a section ring for an effective $\R$-divisor $D$ on a $\Q$-factorial normal projective variety $X$ as $R(D):=\bigoplus_{m \geq 0} H^0(\mathcal{O}_X(\lfloor mD \rfloor))$.
For effective $\R$-divisors $D$ and $E$ on $X$, we write $R(D) \simeq R(E)$ if $H^0(\mathcal{O}_X(\lfloor mD \rfloor)) \simeq H^0(\mathcal{O}_X(\lfloor mE \rfloor))$ for infinitely many integers $m>0$.

\begin{lemma}\label{prop-sd}
Let $X$ be a $\Q$-factorial normal projective variety and $D$ an effective $\Q$-divisor with the $s$-decomposition $D=P_s+N_s$.
\begin{enumerate}
\item[$(1)$] $R(D) \simeq R(P_s)$.
\item[$(2)$] $P_s$ is minimal in the sense that if $L$ is an effective divisor with $R(L) \simeq R(D)$, then $P_s \leq L$.
\item[$(3)$]  Let $E_1, \ldots, E_k$ be mutually distinct prime divisors. Then for any $\epsilon >0$, there is an effective $\R$-divisor $P_0 \sim_{\R} P_s$ such that $\mult_{E_i} P_0 < \epsilon$ for any $i$.
\end{enumerate}
\end{lemma}

\begin{proof}
For (1) and (2), see \cite[Proposition 4.6]{P}.
Let $E$ be a prime divisor on $X$ and fix $\epsilon>0$.
To prove (3), it suffices to find an effective divisor $P_0 \sim_{\Q} P_s$ such that $\mult_{E} P_0 < \epsilon$ since the condition $\mult_{E} \lfloor mP_0 \rfloor < m\epsilon$ is a Zariski open condition in the projective space $|\lfloor mP_0 \rfloor |$ for an integer $m>0$. Suppose that $\mult_{E} P_0 \geq \epsilon$ for every effective divisor $P_0 \sim_{\Q} P_s$. Then $P_s - \epsilon E \geq 0$ and $R(P_s - \epsilon E) \simeq R(P_s)$, which contradicts the minimal property of $P_s$.
\end{proof}

\begin{remark}
Let $D$ be an effective $\R$-divisor on a $\Q$-factorial normal projective variety and $D=P+N$ (resp. $D=P_s + N_s$) the divisorial Zariski decomposition
(resp. the $s$-decomposition).
\begin{enumerate}
\item By Lemmas \ref{lem-divzd} and \ref{prop-sd}, $P_s \leq P$ and $P_s \neq P$ in general. If $D$ is big, then $P_s = P$.
\item $P$ and $P_s$ may be $\R$-divisors even if $D$ is a $\Q$-divisor.
\end{enumerate}
\end{remark}

\subsection{Rational connectedness}

\begin{definition}
Let $X$ be a normal projective variety and $V$ a subset of $X$. We say that $X$ is \emph{rationally connected modulo} $V$ if either
\begin{enumerate}
\item $V \neq \emptyset$ and there is an irreducible component $C$ of $V$ such that for any general point $x\in X$, there exists a rational curve joining $x$ and $C$, or
\item $V = \emptyset$ and $X$ is rationally connected.
\end{enumerate}
If we replace `a rational curve' by `a chain of rational curves', then we say that $X$ is \emph{rationally chain connected modulo} $V$.
\end{definition}

By \cite[Corollary 1.5]{HM2}, if $(X, \Delta)$ is a klt pair,
then $X$ is rationally connected if and only if $X$ is rationally chain connected.
However, it is not true for lc pairs in general (e.g., non-rational lc del Pezzo surface).

It is well known that a Fano type variety is rationally connected (\cite{Z}, \cite{HM2}).
However, a rationally connected variety need not be of Fano type. 

Suppose that $X$ is a smooth projective variety. By \cite[Theorem IV.5.4]{K1}, there is the \emph{maximally rationally connected fibration} (\emph{MRC-fibration} for short) $\pi : X \dashrightarrow Z$. We call $Z$ the \emph{MRC-quotient}. It is easy to see that if $X$ is rationally connected if and only if $Z$ is a point, and $X$ is not uniruled if and only if $Z=X$. Note that by \cite[Corollary 1.4]{GHS}, the MRC-quotient is not uniruled.



\section{Potential pairs}\label{sec-pklt}
In this section, we introduce the notion of potentially klt, lc pairs for $(X,\Delta)$
such that $-(K_X+\Delta)$ are pseudoeffective.
We will see that these potentially klt, lc pairs satisfy many of the properties of the usual klt, lc pairs.

First of all, note that the discrepancy $a(E;X,\Delta)$ depends only on the divisorial valuation $\sigma$ of $X$ such that the center
$\Cent_Y\sigma=E$ is a prime divisor for some high model $Y\to X$.
Thus we will also often write $a(E;X,\Delta)=a(\sigma;X,\Delta)$.

\begin{definition}\label{def-Phi}
Let $(X,\Delta)$ be a pair such that $-(K_X+\Delta)$ is pseudoeffective.
\begin{enumerate}
\item The \emph{potential discrepancy} of the pair $(X,\Delta)$ at a divisorial valuation $\sigma$ of $X$ is defined as
$$
\oa(\sigma;X,\Delta):=a(\sigma;X,\Delta)-\sigmanum(-(K_X+\Delta)).
$$
\item The \emph{total potential discrepancy of the pair} $(X,\Delta)$ is defined as
$$
\mathfrak A(X,\Delta):=\inf_{\sigma}\{\oa(\sigma;X,\Delta)\}
$$
where the infimum is taken over all the divisorial valuations $\sigma$ of $X$.
\item The \emph{total potential discrepancy of the variety} $X$ is defined as
$$
\mathfrak A(X):=\sup_{\Delta}\mathfrak A (X,\Delta)
$$
where the supremum is taken over all effective $\R$-divisors $\Delta$ on $X$ such that
$-(K_X+\Delta)$ are pseudoeffective $\R$-Cartier divisors.
\end{enumerate}
\end{definition}

Potential discrepancies $\oa(\sigma;X,\Delta)$ in (1) can be equivalently defined with the divisorial Zariski decomposition as follows.
For a birational morphism $f:Y\to X$, suppose that $\Cent_X\sigma=E$ is a prime divisor and
let $f^*(-(K_X+\Delta))=P+N$ be the divisorial Zariski decomposition.
Since we have $\sigma_{num}(-(K_X+\Delta))=\sigma_{num}(-f^*(K_X+\Delta))=\mult_EN$,
as in Introduction we can also define
$$\oa(\sigma;X,\Delta):=a(\sigma;X,\Delta)-\mult_EN.$$

\begin{definition}\label{def-pnklt}

Let $(X,\Delta)$ be a pair such that $-(K_X+\Delta)$ is pseudoeffective. Fix $\epsilon \geq 0$.
If $\epsilon=0$, then we omit `$\epsilon$-' in the definitions below.
\begin{enumerate}
\item We say that $(X, \Delta)$ is \emph{$\epsilon$-strictly potentially klt} (resp. \emph{$\epsilon$-strictly potentially lc}) if $\oa(\sigma;X, \Delta) > -1 + \epsilon$ (resp. $\geq -1 + \epsilon$) holds for every divisorial valuation $\sigma$.

\item The \emph{$\epsilon$-strictly potentially non-klt locus} of $(X,\Delta)$ is defined as
$$
\espnklt(X,\Delta):=\bigcup_{\sigma}\Cent_X\sigma
$$
where the union is taken over all $\sigma$ such that $\oa(\sigma;X,\Delta)\leq-1+\epsilon.$

\item We say that $(X, \Delta)$ is \emph{$\epsilon$-potentially klt} (resp. \emph{$\epsilon$-potentially lc}) if $\mathfrak A(X, \Delta) > -1 + \epsilon$ (resp. $\geq -1 + \epsilon$) holds.
\footnote{The potentially klt pair was initially defined as the strictly potentially klt pair.}

\item The \emph{$\epsilon$-potentially non-klt locus} of $(X,\Delta)$ is defined as
$$
\epnklt(X,\Delta):=\bigcap_{\epsilon'>\epsilon}\epsilon'\text{-}\spnklt(X,\Delta).
$$
\end{enumerate}
\end{definition}

By definition, it is easy to see that $\espnklt(X,\Delta)\subseteq\epnklt(X,\Delta)$ for any $\epsilon\geq 0$.
However, it is not clear whether the equality holds in general.
Below, we will mainly consider the case where $\epsilon=0$.
As explained above, we can determine $\pnklt(X,\Delta)$ by considering the divisorial Zariski decompositions of the pull backs of $-(K_X+\Delta)$ on all higher smooth models of $X$.
Unlike the case with the usual klt, lc singularities, we need to consider all the prime divisors over $X$
(not only the prime divisors on a fixed log resolution)
to determine the potential klt, lc pairs.

\begin{lemma}\label{lem-P<1=klt}
Let $(X,\Delta)$ be a pair such that $-(K_X+\Delta)$ is pseudoeffective.
\begin{enumerate}
\item[$(1)$]
If $(X,\Delta)$ is potentially klt (resp. potentially lc), then $(X,\Delta)$ is klt (resp. lc).
\item[$(2)$] $(X,\Delta)$ is potentially klt if and only if $\pnklt(X,\Delta)=\emptyset$.
\end{enumerate}
\end{lemma}

\begin{proof}
(1) Note that $\sigmanum(D)\geq 0$ for any divisor $D$.
Therefore, the statement is obvious since $a(\sigma;X,\Delta)\geq a(\sigma;X,\Delta)-\sigmanum(-(K_X+\Delta))$.

(2) Immediate by definition.
\end{proof}

\begin{remark}\label{rmk-subtle1}
\begin{enumerate}
\item The converse of Lemma \ref{lem-P<1=klt} (1) is easily seen to be false. Any smooth projective variety $X$ is clearly klt, but if $-K_X$ is pseudoeffective and $\lfloor N\rfloor\neq 0$ where $-K_X=P+N$ is the divisorial Zariski decomposition,
then $X$ is not potentially klt. See also Example \ref{ex-pnklt,nklt}.
Note that by definition, if $\mathfrak A(X, \Delta) > -1$, then $(X, \Delta)$ is $\epsilon$-potentially klt for all sufficiently small $\epsilon > 0$.
It is unclear whether the converse with sufficiently small $\epsilon>0$ holds in general.
However, we can show that the converse holds when the pull back $-f^*(K_X+\Delta)$
admits the divisorial Zariski decomposition with nef positive part for some birational morphism $f:Y\to X$.
 See Proposition \ref{cor-frakpklt=pklt}.
\item Note that $\spnklt(X, \Delta)\subseteq \pnklt(X, \Delta)\subseteq\espnklt(X,\Delta)$ for any $\epsilon>0$.
As explained above, if $(X, \Delta)$ is a pair in dimension $2$, then $\spnklt(X, \Delta) = \pnklt(X, \Delta)$.
\end{enumerate}

\end{remark}

As we will see below, potentially klt, lc pairs behave similarly with the usual klt, lc pairs.

\begin{lemma}\label{lem-oa>oa}
Let $(X, \Delta)$ be a pair such that $-(K_X + \Delta)$ is pseudoeffective. If $\Delta'$ is an effective $\R$-Cartier divisor on $X$ such that $-(K_X + \Delta+\Delta')$ is pseudoeffective, then $\oa(\sigma; X, \Delta) \geq \oa(\sigma; X, \Delta+\Delta')$ holds for any divisorial valuation $\sigma$.
In particular, the following hold:
\begin{enumerate}
\item[$(1)$] $\mathfrak A(X,\Delta)\geq\mathfrak A(X,\Delta+\Delta')$.

\item[$(2)$] If $(X,\Delta+\Delta')$ is potentially klt, then so is $(X,\Delta)$.

\item[$(3)$] $\pnklt(X,\Delta)\subseteq\pnklt(X,\Delta+\Delta')$.
\end{enumerate}
\end{lemma}

\begin{proof}
Take a log resolution $f : Y \to X$ of $(X, \Delta+\Delta')$ such that $\Cent_Y \sigma = E$ is a prime divisor. Let $f^*(-(K_X + \Delta))=P+N$ and $f^*(-(K_X+\Delta+\Delta'))=P'+N'$ be the divisorial Zariski decompositions.
Since $P' + (N'+f^*\Delta')=P+N$, it follows from Lemma \ref{lem-divzd} (2) that
$N'+f^*\Delta' \geq N$.
We may write
$$
K_Y = f^*(K_X + \Delta+\Delta') + F = f^*(K_X + \Delta) + (f^*\Delta'+F).
$$
We have $a(E; X, \Delta)=\mult_E (f^*\Delta'+F)$ and $a(E;X,\Delta+\Delta') = \mult_E F$.
Since $f^*\Delta'-N \geq -N'$, we get $f^*\Delta'+F-N \geq F-N'$.
This implies the required inequality
$$
\oa (E; X, \Delta)=\mult_E(f^*\Delta'+F-N)\geq \mult_E(F-N')=\oa(E;X,\Delta+\Delta').
$$
\end{proof}


\begin{lemma}\label{lem-prop of pklt}
Let $(X,\Delta)$ be a pair such that $-(K_X+\Delta)$ is big and $\Delta'$ an effective $\R$-Cartier divisor on $X$.
If $\oa(E;X,\Delta)>-1$ for a prime divisor $E$ over $X$,
then $\oa(E;X,\Delta+\epsilon\Delta')>-1$ for all sufficiently small $\epsilon>0$.
\end{lemma}
\begin{proof}
Note that the asymptotic valuation $D\to \sigmanum(D)$ is continuous on the big cone.
Since the discrepancy $\Delta\to a(E,X,\Delta)$ is a also continuous function, the potential discrepancy
$\oa(E;X,\Delta)$ is also continuous on the big cone.
Thus the statement follows. 
\end{proof}

\begin{remark}\label{rmk-subtle2}
By Lemma \ref{lem-oa>oa}, we have $\pnklt(X,\Delta)\subseteq\pnklt(X,\Delta')$
for any $\Q$-divisor $\Delta'$ such that $\Delta'\geq\Delta$ and $-(K_X+\Delta')$ is a pseudoeffective $\Q$-Cartier divisor.
Even if $\Delta'$ is sufficiently close to $\Delta$, Lemma \ref{lem-prop of pklt}
does not guarantee $\pnklt(X,\Delta)=\pnklt(X,\Delta')$ and
the perturbation into such $\Q$-divisor $\Delta'$ does not seem to follow directly.
However, we will show that this is actually possible in Corollary \ref{cor-r to q}.
\end{remark}

\begin{proposition}\label{prop-intersection pnklt}
Let $(X,\Delta)$ be a pair such that $-(K_X+\Delta)$ is big.
Let $\{\Delta_i\}_{i\geq 1}$ be any sequence of divisors such that
$\Delta_{i+1}\leq\Delta_i$ and $-(K_X+\Delta_i)$ are big $\R$-Cartier divisors for all $i\geq 1$.
If $\lim_{i\to\infty}\Delta_i=\Delta$,
then we have
$$\pnklt(X,\Delta)=\bigcap_{i\geq 1}\pnklt(X,\Delta_i).$$
\end{proposition}

\begin{proof}
It suffices to show that for any $\epsilon >0$, we have
$$
\espnklt(X, \Delta) = \bigcap_{i \geq 1} \espnklt(X, \Delta_i).
$$
The inclusion $\subseteq$ is clear by Lemma \ref{lem-oa>oa}.
Suppose that the strict inclusion $\subsetneq$ holds for some $\epsilon>0$.
Then there exists a divisorial valuation $\sigma$ such that
$$\oa(\sigma;X,\Delta_i)\leq -1+\epsilon <\oa(\sigma;X,\Delta)$$
for any $i\geq 1$.
Note that $a(\sigma;X,\Delta)$ is linear with respect to $\Delta$ and
$\sigma_{num}(\;)$ is continuous on the big cone.
Thus we must have $\oa(\sigma;X,\Delta_i)\to \oa(\sigma;X,\Delta)$ as $i\to \infty$,
which is a contradiction.
\end{proof}



\begin{definition}\label{def-D-map}
Let $X$ be a normal projective variety and $D$ an $\R$-Cartier divisor on $X$.
A birational contraction $f:X\dashrightarrow Y$ is called \emph{$D$-negative} if $f_*D$ is an $\R$-Cartier divisor on a normal projective variety $Y$ and for some common resolution $W$ of $X,Y$
$$
\xymatrix{
&W\ar[ld]_g\ar[rd]^h&\\
X\ar@{-->}[rr]^f&&Y,
}
$$
we have
$$
g^*D=h^*f_*D+\Gamma
$$
where $\Gamma$ is an $h$-exceptional effective divisor whose support contains all the strict transforms of $f$-exceptional prime divisors.
\end{definition}

The following is one of the most typical $D$-negative birational maps.

\begin{definition}\label{def-D-MMP}
Let $D$ be a pseudoeffective $\R$-Cartier divisor on a normal projective variety $X$.
A birational map $\varphi:X\dashrightarrow Y$ (or just the variety $Y$) is called a \emph{$D$-minimal model} of $X$
if the map $\varphi$ is $D$-negative and the proper transform $D_Y:=f_*D$ is nef.
\end{definition}



The following observation (Proposition \ref{prop-invariant oa} and Corollary \ref{cor-MMP on pnklt})
is what inspired us to define the `potentially' klt, lc pairs.

\begin{proposition}\label{prop-invariant oa}
Let $(X,\Delta)$ be a pair such that $-(K_X+\Delta)$ is pseudoeffective.
If a birational contraction map $f:X\dashrightarrow Y$ is $-(K_X  + \Delta)$-negative, then $\oa(\sigma;X,\Delta)=\oa(\sigma;Y, f_*\Delta )$ for any divisorial valuation $\sigma$. In particular, $(X,\Delta)$ is potentially klt (resp. potentially lc) if and only if so is $(Y, f_*\Delta)$.
\end{proposition}

\begin{proof}
Let $\Delta_Y := f_*\Delta$.
Fix a common log resolution $W$ of $(X,\Delta)$ and $(Y,\Delta_Y)$
$$
\xymatrix{
&W\ar[ld]_g\ar[rd]^h&\\
X\ar@{-->}[rr]^f &&Y,
}
$$
and write
$$
g^*(-(K_X+\Delta))=h^*(-(K_Y+\Delta_Y))+\Gamma
$$
where $\Gamma$ is an $h$-exceptional effective divisor given by
$$
\Gamma=\sum_{E\text{ on }W}\big(a(E;X,\Delta)-a(E;Y,\Delta_Y)\big)E.
$$
We may assume that $G:=\Cent_W \sigma$ is a prime divisor on $W$.
Let $h^*(-(K_Y+\Delta_Y))=P+N$ be the divisorial Zariski decomposition.
By \cite[Lemma III.5.14]{Na}, $g^*(-(K_X+\Delta))=P + (N+\Gamma)$ is the divisorial Zariski decomposition.
Thus we have
$$
\begin{array}{rl}
\oa(\sigma;X,\Delta)&=a(G;X,\Delta)-\sigmanum(-(K_X+\Delta))\\
&=a(G;X,\Delta) - \mult_G (N+\Gamma) \\
&=a(G;X,\Delta) - \mult_G N - \big(a(G;X,\Delta)-a(G;Y,\Delta_Y)\big)\\
&=a(G;Y,\Delta_Y) - \sigmanum(-(K_Y+\Delta_Y)) \\
&= \oa(G;Y,\Delta_Y).
\end{array}
$$
\end{proof}

\begin{corollary}\label{cor-MMP on pnklt}
Let $(X,\Delta)$ be a pair such that $-(K_X+\Delta)$ is pseudoeffective and
$\varphi:X\dashrightarrow Y$ is a $-(K_X+\Delta)$-minimal model of $(X,\Delta)$.
Then $(X,\Delta)$ is potentially klt if and only if $(Y,\Delta_Y)$ is klt.
\end{corollary}

\begin{proof}
Since $\varphi$ is a $-(K_X+\Delta)$-negative birational map and $\sigmanum(-(K_Y + \Delta_Y))=0$ for every divisorial valuation $\sigma$ of $Y$, the assertion immediately follows from Proposition \ref{prop-invariant oa}.
\end{proof}

\begin{proposition}\label{cor-frakpklt=pklt}
Let $(X,\Delta)$ be a pair such that $-(K_X+\Delta)$ is pseudoeffective. Suppose that $-f^*(K_X+\Delta)$ admits the divisorial Zariski decomposition with nef positive part for some birational morphism $f: Y \to X$.
Then
$$
\spnklt(X, \Delta)=\pnklt(X, \Delta).
$$
In particular, $(X,\Delta)$ is potentially klt if and only if $(X,\Delta)$ is strictly potentially klt.
\end{proposition}

\begin{proof}
Let $-f^*(K_X+\Delta) = P+N$ be the divisorial Zariski decomposition such that $P$ is nef.
We write
$$
-K_Y = f^*(-(K_X+\Delta)) + E = P+N+E.
$$
For any birational morphism $g: Z \to Y$, we write
$$
-K_Z = g^*(-K_Y)+F = g^*f^*(-(K_X+\Delta)) + g^*E+F = g^*P+g^*N + g^*E+F.
$$
Note that $g^*f^*(-(K_X+\Delta)) = g^*P+g^*N$ is also the divisorial Zariski decomposition.
We now may assume that $g^*N + g^*E+F$ has an snc support. 
Then for any $\epsilon \geq 0$, we get
$$
\espnklt(X, \Delta)=f \circ g ( \text{Supp}( (g^*N+g^*E+F)^{\geq 1-\epsilon} ) ).
$$
Since we have 
$$
 (g^*N+g^*E+F)^{\geq 1}  =  (g^*N+g^*E+F)^{\geq 1-\epsilon} 
$$
for a sufficiently small $\epsilon>0$, it follows that  $\spnklt(X, \Delta) = \espnklt(X, \Delta)$. Thus 
$\spnklt(X, \Delta)=\pnklt(X, \Delta)$.
\end{proof}

The following gives an example of a non-Fano type Mori dream space $X$ such that $-K_X$ is movable and big.
We will observe that the $-K_X$-minimal model program indeed creates singularities worse than klt when $(X,\Delta)$ is not potentially klt.

\begin{example}\label{semifano example}
Let $\overline{S}$ be a rational $\Q$-homology projective plane with numerically trivial anticanonical divisor not containing rational double points (see \cite[Example 6.3]{HP1} for examples). By \cite[Theorem 1.5]{HP1}, the minimal resolution $S$ of $\overline{S}$ is a Mori dream space. Note that $S$ is of Calabi-Yau type. Consider the smooth rational variety $X:=\P(\mathcal{O}_S(A) \oplus \mathcal{O}_S(A) \oplus \mathcal{O}_S(-K_S-2A))$ where $A$ is a very ample divisor on $S$. By \cite[Proposition 2.6]{CG}, $X$ is a Mori dream space, and by \cite[Proposition 2.10]{CG}, $X$ is of Calabi-Yau type. Note that $X$ is not of Fano type since $\kappa(-K_S)=0$.

We now show that $-K_X$ is movable and big. Since $mA + (-K_S - 2A)$ can be big for some sufficiently large integer $m>0$, the tautological line bundle $\mathcal{O}_X(1)$ is big. Thus $\mathcal{O}_X(-K_X)=\mathcal{O}_X(3)$ is also big. If we regard $X$ as locally a trivial $\P^2$-bundle over $S$, then the global section of $H^0(\mathcal{O}_X(1))$ is locally given by $s_1x_1 + s_2x_2$ where $x_0, x_1, x_2$ are coordinates for $\P^2$ and $s_1, s_2$ are global sections of $H^0(\mathcal{O}_S(A))$. Since the base locus of $|\mathcal{O}_X(1)|$ is locally $V(x_0)$, it follows that $-K_X$ is movable.

Since $X$ is a Mori dream space, we can run a $-K_X$-minimal model program $X \dashrightarrow X'$, which is a small modification to a $\Q$-factorial variety $X'$ such that $-K_{X'}$ is nef and big. Since $X$ is not of Fano type, it follows that $X'$ contains singularities worse than log terminal.
Note that $X'$ is also of Calabi-Yau type, and hence, $X'$ has lc singularities.
Thus, $\mathfrak A (X) = -1$ and $X$ is not potentially klt.
\end{example}

\section{Potentially non-klt locus}\label{sec-pnklr locus}
In this section, we prove some basic properties of potentially non-klt locus. In particualr, we prove Proposition \ref{main ingredient}.

\begin{lemma}\label{lem-nklt+nnef=pnklt}
Let $(X,\Delta)$ be a pair such that $-(K_X+\Delta)$ is pseudoeffective.
Then we have
$$\nklt(X, \Delta) \subseteq \pnklt(X,\Delta) \subseteq \nklt(X,\Delta)\cup \nnef(-(K_X+\Delta)).$$
\end{lemma}

\begin{proof}
If $a(\sigma;X,\Delta)\leq-1$, then
$$
\oa(\sigma;X,\Delta)=a(\sigma;X,\Delta)-\sigmanum(-(K_X+\Delta))\leq a(\sigma;X,\Delta)\leq-1.
$$
Thus $\nklt(X,\Delta)\subseteq \pnklt(X,\Delta)$.

Assume that $V$ is an irreducible component of $\pnklt(X,\Delta)$, not contained in $\nklt(X,\Delta)$.
Then there exists a divisorial valuation $\sigma$ such that $\Cent_X\sigma=V$ and
the following are satisfied:
$$
a(\sigma;X,\Delta)-\sigmanum(-(K_X+\Delta))\leq -1+\epsilon \text{ \;\;and\;\; } a(\sigma;X,\Delta)>-1+\epsilon>-1
$$
for some sufficiently small $\epsilon>0$.
Thus $0<a(\sigma;X,\Delta)+1 - \epsilon \leq\sigmanum(-(K_X+\Delta))$ so that $V\subseteq \nnef(-(K_X+\Delta))$.
\end{proof}

The following examples show that the inclusions in Lemma \ref{lem-nklt+nnef=pnklt} can be
strict.

\begin{example}\label{ex-pnklt,nklt}
Let $Y$ be a normal projective surface such that $-K_Y$ is ample $\Q$-Cartier divisor.
Let $f: X \to Y$ be the minimal resolution and $-K_X=P+N$ the Zariski decomposition. Note that $f$ is the anticanonical morphism and $P=f^*(-K_Y)$ (see \cite{HP2}).

\begin{enumerate}[leftmargin=0cm,itemindent=.6cm]
\item Suppose that $Y$ is a klt del Pezzo surface.
Then $(X, N)$ is a klt weak del Pezzo pair by \cite[Theorem 1.1]{HP2} so that
$$\emptyset = \nklt(X,N) = \pnklt(X,N) = \nnef(-(K_X+N)) \cup  \nklt(X,N).$$

\item Suppose that $Y$ is a klt del Pezzo surface and $N\neq 0$. Then we can see that
$$\emptyset = \nklt(X,0) = \pnklt(X,0) \subsetneq \nnef(-K_X) \cup \nklt(X,0) = \Supp N.$$


\item Suppose that $Y$ is a non-rational lc del Pezzo surface. Then $N$ consists of single component with coefficient 1 by \cite[Theorem 1.6]{HP2} so that
$$\emptyset = \nklt(X,0) \subsetneq \pnklt(X,0) = \nnef(-K_X) \cup \nklt(X,0) = \Supp N.$$
We remark that $X$ is rationally connected modulo $N$.

\item Suppose that $X$ is rational and $\nklt(Y,0)\neq \emptyset$ (for the existence of such surfaces, see \cite[Example 4.3]{HP1}).
There is an irreducible component $E$ of $N$ with $\mult_E N\geq 1$.
Then $E \subseteq \pnklt(X,0)$.
On the other hand, for any irreducible component $F$ of $N$ with $\mult_F N<1$, we can easily see that $F \nsubseteq \pnklt(X,0)$, but $F \subseteq \nnef(-K_X)= \Supp(N)$. Thus we have
$$\emptyset = \nklt(X,0) \subsetneq \pnklt(X,0) \subsetneq \nnef(-K_X) \cup \nklt(X,0).$$
\end{enumerate}
\end{example}

\begin{lemma}\label{lem-fund ni}
Let $(X, \Delta)$ be a pair such that $-(K_X+\Delta)$ is pseudoeffective.
Suppose that $D$ is an effective $\R$-Cartier divisor $D$ on $X$ such that
for any resolution $f: Y \to X$, we have $f^*D \geq N$ where $f^*(-(K_X+\Delta))=P+N$ is the divisorial Zariski decomposition.
Then $\pnklt(X, \Delta)\subseteq\espnklt(X, \Delta)\subseteq\nklt(X,\Delta+D)$
for any sufficiently small $\epsilon>0$.
\end{lemma}

\begin{proof}
The first inclusion is by definition.
For the second inclusion, it suffices to show that for any divisorial valuation $\sigma$,
we have $a(\sigma ; X, \Delta+D)\leq\oa(\sigma ; X, \Delta)$.
By taking a log resolution $f : Y \to X$ of $(X, \Delta+ D)$ such that $G:=\Cent_Y \sigma$ is a prime divisor on $Y$, we can write
$$
K_Y + \Gamma = f^*(K_X + \Delta) + E
$$
where $\Gamma$ and $E$ are effective divisors having no common components.
Then we have
$$
\oa(\sigma ; X, \Delta)= a(\sigma; X, \Delta) - \mult_G N =\mult_G (-\Gamma + E - N).
$$
On the other hand, we have
$$
K_Y + \Gamma + f^*D = f^*(K_X + \Delta + D) + E,
$$
thus we obtain
$$
a(\sigma ; X, \Delta+D)=\mult_G(-\Gamma+E-f^*D).
$$
By the assumption $f^*D \geq N$,
it follows that
$$a(\sigma ; X, \Delta+D)=\mult_G(-\Gamma+E-f^*D)\leq\mult_G (-\Gamma + E - N)=\oa(\sigma ; X, \Delta).$$
\end{proof}

The following is a fundamental property of potentially non-klt locus.

\begin{proposition}[{cf. \cite[Proposition 3.3]{BP}}]\label{prop-pnkltQ}
Let $(X, \Delta)$ be a pair such that $-(K_X + \Delta)$ is a big $\Q$-Cartier divisor.
Then there is an effective $\Q$-Cartier divisor $D$ such that
$D \sim_{\Q} -(K_X+\Delta)$ and $\nklt(X, \Delta+D)=\pnklt(X, \Delta)$.
\end{proposition}

\begin{proof}
We claim that for a sufficiently small $\epsilon>0$, there exists an effective $\Q$-divisor $D_{\epsilon}\sim_\Q -(K_X+\Delta)$ such that
\begin{equation}\tag{$\#$}\label{sharp}
\nklt(X, \Delta+D_{\epsilon}) \subseteq \espnklt(X, \Delta).
\end{equation}

Assuming the validity of this claim for the moment, we prove our statement first.
By Lemma \ref{lem-fund ni}, $\epsilon'\text{-}\spnklt(X, \Delta) \subseteq \nklt(X, \Delta+D_{\epsilon})$ holds for any sufficiently small $\epsilon>\epsilon'>0$. Thus using the claim ($\#$), we can form decreasing sequences 
$$
\{\nklt(X,\Delta+D_i)\}_{i\in\N} \text{ and } \{ \epsilon_i\text{-}\spnklt(X,\Delta)\}_{i\in\N}
$$ 
such that 
$$
\nklt(X,\Delta+D_{i+1})\subseteq\epsilon_{i+1}\text{-}\spnklt(X,\Delta)\subseteq \nklt(X,\Delta+D_{i})\subseteq\epsilon_{i}\text{-}\spnklt(X,\Delta)
$$
with $\epsilon_i\to0$ as $i\to \infty$.
Note that each $\nklt(X, \Delta+D_i)$ is a closed set.
By the Noetherian property,
$\nklt(X, \Delta+D_i)$ stabilizes as a closed set, and so does $\epsilon_i\text{-}\spnklt(X,\Delta)$.
Thus we get $\nklt(X, \Delta+D_i)=\pnklt(X, \Delta)$ for a sufficiently large $i$.
By letting $D:=D_{i}$, we conclude that the statement of proposition holds.

It only remains to prove the claim ($\#$).
Fix a sufficiently small $\epsilon >0$ and consider a resolution $g : Y \to X$ of singularities.
By \cite[Theorem 11.2.21]{posII}, there exists an effective $\Z$-divisor $G'$ on $Y$ such that
\begin{equation}\tag{!}\label{!}
\mc J(||mg^*(-(K_X+\Delta))||)\otimes \mathcal{O}_Y(-G') \subseteq  \frak b(|mg^*(-(K_X+\Delta))|)
\end{equation}
for a sufficiently large and divisible integer $m >0$.
Take a log resolution $h : Z \to Y$ of $(Y, G')$ so that $G:=h^*G'$ has a simple normal crossing support.
Let $\phi:=g \circ h : Z \to X$.
Now consider a log resolution $f : W \to Z$ of both the base ideal $\frak b(|m\phi^*(-(K_X+\Delta))|)$
and the asymptotic multiplier ideal $\mc J(||m\phi^*(-(K_X+\Delta))||)$ for some sufficiently divisible and large integer $m>0$.
Thus we obtain the following birational morphisms
\[
\begin{split}
\xymatrix{
 W \ar@/^1.7pc/[rrr]^{\phi \circ f}\ar[r]^{f} & Z \ar[r]^{h} \ar@/_1.4pc/[rr]_{\phi} & Y \ar[r]^{g} & X 
}
\end{split}.
\]
By Hironaka, we can assume that $f$ is a composition of blow-ups at smooth centers.
Arguing as in \cite[Proof of Lemma 9.2.19]{posII}, we can check that there is a constant $c>0$ independent of $f$ such that the multiplicities of $f^*G$ at any prime divisors are smaller than $c$.
By taking $m>0$ sufficiently large, we may also assume that
$$
\text{multiplicity of $\frac{1}{m}f^*G$ at any prime divisor}< \frac{c}{m} < \frac{\epsilon}{2}.
$$
We now write
$$
f^{-1}\frak b(|m\phi^*(-(K_X+\Delta))|) \cdot \mathcal{O}_W = \mathcal{O}_W(-E_m)
$$
and
$$
f^{-1}\mc J(||m\phi^*(-(K_X+\Delta))||) \cdot \mathcal{O}_W = \mathcal{O}_W(-F_m)
$$
for some effective divisors $E_m,F_m$ on $W$.
We obtain a base point free divisor $M_m := mf^* \phi^*(-(K_X+\Delta)) - E_m$ on $W$.
Let $f^* \phi^*(-(K_X+\Delta))  = P+N$ be the divisorial Zariski decomposition.
By Lemma \ref{lem-divzd} (2) and \cite[Proposition 2.5]{ELMNP2}, we have
$$
F_m \leq mN \leq E_m.
$$
On the other hand, by (\ref{!}), $E_m \leq F_m + f^*G$. Thus for every prime divisor $\mathbf D$ on $W$, we have
$$
\mult_{\mathbf D} \left( \frac{1}{m} E_m - N \right) \leq \mult_{\mathbf D} \left( \frac{1}{m} E_m - \frac{1}{m} F_m \right) \leq \mult_{\mathbf D} \frac{1}{m}f^*G < \frac{\epsilon}{2}.
$$
For a sufficiently large and divisible integer $m' >m$, we also have a fixed divisor $E_{m'}$ of the linear system $|m'f^*\phi^*(-(K_X+\Delta))|$.
Since $\frac{1}{m}E_m \geq \frac{1}{m'}E_{m'} \geq N$, it follows that
\begin{equation}\tag{*}\label{*}
0\leq \mult_{\mathbf D} \left( \frac{1}{m} E_m - \frac{1}{m'}E_{m'} \right) \leq \mult_{\mathbf D} \left( \frac{1}{m} E_m - N \right)<\frac{\epsilon}{2}
\end{equation}
for every prime divisor $\mathbf D$ on $W$.
Now choose a general member $M_0 \in |M_m|$ and let
$$
D_{\epsilon}:=(\phi \circ f)_*\left(\frac{1}{m}M_0 + \frac{1}{m}E_m\right).
$$
Note that $D_{\epsilon} \sim_\Q -(K_X+\Delta)$. We will see that our claim (\ref{sharp}) holds with our choice of $D_{\epsilon}$.
We may assume that $\phi \circ f : W \to X$ is a log resolution of $(X, \Delta+D_{\epsilon})$.
For some divisor $\widetilde{E}$ on $W$, we may write
$$
-K_W =(\phi \circ f)^*(-(K_X+\Delta)) + \widetilde{E} = 
 (\phi \circ f)^*(-(K_X+\Delta+D_{\epsilon})) + \frac{1}{m}M_0 +  \widetilde{E}+\frac{1}{m}E_m.
$$
Suppose $\mathbf D$ is a prime divisor on $W$ such that
$$
a(\mathbf D;X,\Delta+D_\epsilon)=\mult_{\mathbf D}\left(-\widetilde E-\frac{1}{m}E_m\right)\leq -1.
$$
By Lemma \ref{lem-divzd} (3), $\mult_{\mathbf D} (P) < \frac{\epsilon}{2}$, and $\frac{1}{m'}E_{m'} - N \leq P$ for a sufficiently divisible and large integer $m'>0$. Thus we get
\begin{equation}\tag{**}\label{**}
\mult_{\mathbf D} \left( \frac{1}{m'}E_{m'} - N \right) < \frac{\epsilon}{2}.
\end{equation}
By (\ref{*}) and (\ref{**}), we obtain
$$
\begin{array}{rl}
\oa(\mathbf D;X,\Delta)&=\mult_{\mathbf D}(-\widetilde E-N)\\
&=\mult_{\mathbf D}\{(-\widetilde E-\frac{1}{m}E_m)+(\frac{1}{m}E_m-\frac{1}{m'}E_{m'})+(\frac{1}{m'}E_{m'}-N)\}\\
&<-1+\frac{\epsilon}{2}+\frac{\epsilon}{2}\\
&=-1+\epsilon.
\end{array}
$$
This implies our claim (\ref{sharp}).
\end{proof}

In particular, $\pnklt(X,\Delta)$ is a Zariski closed subset of $X$.

\begin{corollary}\label{cor-r to q}
Let $(X, \Delta)$ be a pair such that $-(K_X + \Delta)$ is big.
Then there exists an effective $\Q$-divisor $\Delta'$
such that $-(K_X+\Delta')$ is big, $\Delta'\geq \Delta$ and $\pnklt(X,\Delta)=\pnklt(X,\Delta')$.
\end{corollary}
\begin{proof}
We can easily take a sequence $\{\Delta_i\}$ of $\Q$-divisors such that $-(K_X+\Delta_i)$ are big $\Q$-Cartier divisors, $\Delta_{i+1}\leq\Delta_i$ and $\lim_{i\to\infty}\Delta_i=\Delta$.
By Proposition \ref{prop-intersection pnklt}, we have
$$
\pnklt(X,\Delta)=\bigcap_i\pnklt(X,\Delta_i).
$$
As we have seen above, the sets $\pnklt(X,\Delta_i)$ are closed for all $i$.
Note also that $\pnklt(X,\Delta_{i+1})\subseteq\pnklt(X,\Delta_i)$ for all $i$ by Lemma \ref{lem-oa>oa}.
Thus by the Noetherian property, $\pnklt(X,\Delta)=\pnklt(X,\Delta_i)$ for all $i\gg0$.
Now take $\Delta'=\Delta_i$ for any $i\gg 0$.
\end{proof}

\begin{proposition}\label{prop-pnkltR}
Let $(X, \Delta)$ be a pair such that $-(K_X + \Delta)$ is a big $\R$-Cartier divisor.
Then there is an effective $\R$-Cartier divisor $D$ such that
$D \sim_{\R} -(K_X+\Delta)$ and $\nklt(X, \Delta+D)=\pnklt(X, \Delta)$.
\end{proposition}
\begin{proof}
We have proven the $\Q$-divisor case.
Suppose that $\Delta$ is an $\R$-divisor.
Then by Corollary \ref{cor-r to q},
there exists a $\Q$-divisor $\Delta'$ such that $\Delta'\geq\Delta$ and $\pnklt(X,\Delta)=\pnklt(X,\Delta')$.
Let $\Delta''=\Delta'-\Delta$.
Then by Proposition \ref{prop-pnkltQ}, there exists an effective $\Q$-divisor $D'$ such that
$D'\sim_\Q-(K_X+\Delta')$ and $\nklt(X,\Delta'+D')=\pnklt(X,\Delta')$.
We can easily check that $D=D'+\Delta''$ satisfies the required properties.
\end{proof}

Recall that the non-nef locus is not Zariski closed in general (see \cite{John}).
It is unclear whether $\nklt(X, \Delta) \cup \nnef(-(K_X+\Delta))$ is Zariski closed.

\begin{corollary}\label{cor-pnklt big closed}
Let $(X, \Delta)$ be a pair such that $-(K_X + \Delta)$ is a big $\R$-Cartier divisor.
Then $\pnklt(X, \Delta)$ is a Zariski closed subset of $X$.
\end{corollary}
\begin{proof}
It is immediate by Proposition \ref{prop-pnkltR}.
\end{proof}

We point out that Propositions \ref{prop-pnkltQ} and \ref{prop-pnkltR} do not hold in general
if $-(K_X+\Delta)$ is only pseudoeffective.

\begin{example}\label{nikex}
Let $S$ be a smooth projective rational surface as in $(***)$ of \cite[p.84]{Nik}).
Then $-K_S$ is nef and $\kappa(-K_S)=0$. Note that $|-mK_S|$ consists of a single element $mD$ for every $m>0$. Furthermore, there is an irreducible curve $C$ on $S$ such that $\mult_C D >1$.
Thus $S$ is not of Calabi-Yau type.
Furthermore, we note that $\nklt(S, D) \neq \emptyset$ while $\pnklt(S, 0) = \emptyset$.
\end{example}


Even if $-(K_X+\Delta)$ is only pseudoeffective, we still have the following

\begin{proposition}\label{pnklt-locus-pseff}
Let $(X, \Delta)$ be a pair with an effective $\R$-divisor $\Delta$ on $X$ such that $-(K_X + \Delta)$ is a pseudoeffective $\R$-Cartier divisor.
For any ample $\R$-divisor $A$ such that $-(K_X+\Delta)+A$ is big $\Q$-Cartier, there is an effective $\Q$-Cartier divisor $D$ such that $D \sim_{\Q} -(K_X+\Delta)+A$ and $\nklt(X,\Delta+D)\subseteq\pnklt(X,\Delta)$.
\end{proposition}

\begin{proof}
We can take an ample $\R$-divisor $A$ such that $-(K_X+\Delta)+A$ is big $\Q$-Cartier. Then there is an effective $\Q$-divisor $E \sim_{\Q} -(K_X+\Delta)+A$.
Consider a resolution $g : Y \to X$ of singularities. Let
$g^*(-(K_X+\Delta))=P+N$ and $g^*(-(K_X+\Delta) + A) = P'+N'$ be divisorial Zariski decomposition. Since $P+f^*A$ is movable and big, it follows from Lemma \ref{lem-divzd} (2) that $N \geq N'$.
Now by applying the argument of the proof of Proposition \ref{prop-pnkltQ} to $-(K_X+\Delta) + A$ instead of $-(K_X+\Delta)$, we can find an effective $\Q$-Cartier divisor $D$ such that $D \sim_{\Q} -(K_X+\Delta)+A$ and $\nklt(X, \Delta+D) \subseteq \pnklt(X, \Delta)$.
\end{proof}

It is unknown whether $\pnklt(X,\Delta)$ is also Zariski closed as in Corollary \ref{cor-pnklt big closed}
if $-(K_X+\Delta)$ is only pseudoeffective.

\section{Characterization of Fano type varieties}\label{sec-characteirze FT}
In this section, we characterize the varieties of Fano type using the properties of potentially
klt, lc pairs in Theorems \ref{thrm-charft} and \ref{main thrm1}.
We also give a characterization of klt Calabi-Yau type varieties by using a variant of potential pairs defined with the $s$-decomposition (Theorem \ref{thrm-charcy}).

\begin{theorem}\label{thrm-charft}
For a $\Q$-factorial normal projective variety $X$, the following are equivalent:
\begin{enumerate}
 \item[$(1)$] $X$ is of Fano type.
 \item[$(2)$] $\mathfrak A (X) >-1$ and $-K_X$ is big.
 \item[$(3)$] There is an effective $\R$-divisor $\Delta$ such that $(X, \Delta)$ is potentially klt and $-K_X$ is big.
 \item[$(4)$] There is an effective $\Q$-divisor $\Gamma$ such that $(X, \Gamma)$ is potentially klt and $-(K_X+\Gamma)$ is big.
\end{enumerate}
\end{theorem}

\begin{proof}
(1)$\Rightarrow$(2):
If $X$ is a variety of Fano type, then there exists an effective divisor $\Delta$ on $X$ such that
$(X,\Delta)$ is klt and $-(K_X+\Delta)$ is ample. It is clear that $-K_X$ is big.
For any divisorial valuation $\sigma$ of $X$, we have $\sigmanum(-(K_X+\Delta))=0$ so that $\oa(\sigma;X, \Delta)=a(\sigma;X, \Delta)$. Thus we obtain
$$
\mathfrak A(X,\Delta)=\inf_\sigma \oa(\sigma;X,\Delta)=\inf_\sigma a(\sigma; X, \Delta) >-1.
$$
Therefore, $\mathfrak A (X)>-1$.

(2)$\Rightarrow$(3):
By definition, there is an effective $\R$-divisor $\Delta$ on $X$ such that $\mathfrak A (X, \Delta) >-1$. Thus, $(X, \Delta)$ is potentially klt.

(3)$\Rightarrow$(4): Take $\Gamma=0$ and apply Lemma \ref{lem-oa>oa}.

(4)$\Rightarrow$(1):
By Proposition \ref{prop-pnkltQ}, there exists an effective $\Q$-Cartier divisor $D$ such that $D\sim_\Q -(K_X+\Gamma)$ and $\nklt(X, \Gamma+D)=\pnklt(X, \Gamma)=\emptyset$.
In particular, $(X, \Gamma+D)$ is a klt Calabi-Yau pair with big $-K_X$.
It is easy to check that such variety $X$ is of Fano type.
\end{proof}

In fact, we can drop the $\Q$-factoriality condition since we can always take a small
$\Q$-factorialization for any klt pair (see \cite[Corollary 1.4.3]{BCHM}).
The details are left to the readers.

\begin{corollary}\label{cor-FT=pklt(X,N)}
Let $X$ be a $\Q$-factorial normal projective variety such that $-K_X$ is big
and $-K_X=P+N$ the divisorial Zariski decomposition. Then $X$ is of Fano type if and only if $(X, N)$ is potentially klt.
\end{corollary}
\begin{proof}
The `if' direction follows from Theorem \ref{thrm-charft}.
For the `only if' direction, suppose that $X$ is of Fano type.
There exists an effective $\R$-divisor $\Delta$ on $X$ such that $(X,\Delta)$ is klt and $-(K_X+\Delta)$ is ample. Since $\Delta \geq N$ by Lemma \ref{lem-divzd} (2), it follows from Lemma \ref{lem-oa>oa} that $\mathfrak A(X, N) \geq \mathfrak A(X, \Delta)$.
Since $-(K_X+\Delta)$ is ample, we have
$$\mathfrak A(X,\Delta)=\inf_{\sigma}\oa(\sigma,X,\Delta)=\inf_{\sigma}a(\sigma,X,\Delta)>-1.$$
Thus $(X, N)$ is potentially klt.
\end{proof}

Note that `potentially klt' in Corollary \ref{cor-FT=pklt(X,N)} can be replaced by the usual `klt' if $\dim X=2$.

\begin{remark}\label{rmk-suspect CY}
One may suspect whether
\emph{a normal projective variety $X$ is of Calabi-Yau type if and only if $\mathfrak A (X) \geq -1$ and $-K_X$ is pseudoeffective}. The `only if' part is trivial.
However, the `if' part is easily seen to be false.
Example \ref{nikex} gives a counterexample.
There also exists a variety $X$ such that $\mathfrak A (X) \geq -1$ and $-K_X$ is big, but $X$ is not of Calabi-Yau type (see Example \ref{lcweakfano}).
However, we feel that it is reasonable to impose the bigness condition on $-K_X$ to study the varieties
$X$ with $\mathfrak A (X) \geq -1$.
\end{remark}

\begin{example}\label{lcweakfano}
There exists an lc weak Fano pair $(X, \Delta)$ such that $X$ is not of Calabi-Yau type (see \cite[Example 6.2]{CHP}).
Since $(X,\Delta)$ is lc and $-(K_X + \Delta)$ is nef and big, we have $\oa(\sigma; X, \Delta)=a(\sigma; X, \Delta) \geq -1$ for any divisorial valuation $\sigma$ of $X$.
Thus $\mathfrak A (X) \geq -1$.

On the other hand, there also exists a variety $X'$ of Calabi-Yau type such that there is no boundary divisor $\Gamma$
for which $(X', \Gamma)$ is an lc weak Fano pair (see \cite[Example 6.3]{CHP}).
Note that we still have $\mathfrak A (X') \geq -1$.
\end{example}

We ask the following.

\begin{problem}
Characterize a normal projective variety $X$ such that $-K_X$ is big and $\mathfrak A (X) \geq -1$.
\end{problem}

In the surface case, this can be answered without difficulty.

\begin{proposition}
Let $X$ be a normal projective surface such that $-K_X$ is big. Then $X$ is of Calabi-Yau type if and only if $\mathfrak A(X) \geq -1$.
\end{proposition}

\begin{proof}
If $(X, \Delta)$ is an lc Calabi-Yau pair, then $\mathfrak A(X, \Delta) \geq -1$ so that $\mathfrak A(X) \geq -1$. Suppose that $\mathfrak A(X) \geq -1$. Since the anticanonical morphism $f: X \to Y$ is a $-K_X$-minimal model, it follows from Proposition \ref{prop-invariant oa} that $(Y,0)$ is lc, i.e., $Y$ is an lc del Pezzo surface. By \cite[Theorem 3.2.1]{CHP}, $X$ is of Calabi-Yau.
\end{proof}


At this moment, characterizing Calabi-Yau type varieties using the current potential pairs seems to be a hard problem.

\begin{remark}\label{rmk-variant pot sing}
Recall that the notions of potentially klt, lc pairs can be defined with the divisorial Zariski decomposition.
In fact, similar notions can be also defined with other generalizations of Zariski decompositions.
Below we will briefly develop a variant of potential pairs using the $s$-decomposition, instead of divisorial Zariski decomposition.
We will also demonstrate how such variation of the potentially singular pairs can be used to characterize \emph{\textbf{klt}} Calabi-Yau type varieties.
\end{remark}

\begin{definition}[cf. Definition \ref{def-Phi}]\label{def-s-Phi}
Let $(X, \Delta)$ be a pair such that $\Delta$ is a $\Q$-divisor and
$-(K_X + \Delta)$ is effective (up to $\sim_\Q$).
For a divisorial valuation $\sigma$ of $X$, take a resolution $f : Y \to X$ such that the center $\Cent_Y \sigma=E$ is a prime divisor.
Consider the $s$-decomposition $f^*(-(K_X + \Delta)) = P_s + N_s$.
We define the following numbers:
\begin{enumerate}
\item (\emph{potential $s$-discrepancy of $(X,\Delta)$})
$$\oa_s(\sigma; X, \Delta) := a(\sigma; X, \Delta) - \mult_E N_s.$$
\item (\emph{total potential $s$-discrepancy})
$$\mathfrak A_s(X,\Delta):=\inf_{\sigma}\{\oa_s(\sigma;X,\Delta)\}$$
 where the infimum is taken over all the divisorial valuations $\sigma$ of $X$.
\item (\emph{potential $s$-discrepancy of $X$})
$$\mathfrak A_s(X):=\sup_{\Delta}\mathfrak A_s (X,\Delta)$$
where the supremum is taken over all effective $\Q$-divisors $\Delta$ on $X$ such that
$-(K_X+\Delta)$ is effective.
\end{enumerate}
\end{definition}

The following lemma implies that the potential $s$-discrepancy $\oa_s(\sigma; X, \Delta)$
can be computed on any sufficiently high model $Y\to X$ with $Y$ smooth.
Thus the potential $s$-discrepancies are well defined.

\begin{lemma}\label{cons-neg-s-decomp}
Let $f : Y \to X$ be a birational morphism of $\Q$-factorial normal projective varieties.
If $D$ is an effective $\Q$-Cartier $\Q$-divisor with the $s$-decompositions $D=P_s+ N_s$ and $f^*D = P_s'+N_s'$, then $\mult_{f^{-1}_*E}N_s' = \mult_E N_s$ for any irreducible component $E$ of $N_s$.
\end{lemma}

\begin{proof}
Since $N_s' \geq f^*N_s$ by Lemma \ref{prop-sd} (2), it follows that
$$
\mult_{f^{-1}_*E}N_s' \geq \mult_{f^{-1}_*E} f^*N_s=\mult_E N_s.
$$
Suppose that $\mult_{f^{-1}_*E}N_s' > \mult_E N_s$. Take a rational number $\epsilon>0$ such that $\mult_{f^{-1}_*E}N_s' - \mult_E N_s > \epsilon$. Then $f^*P_s - \epsilon f^{-1}_* E \geq P_s'$ and $R(P_s') \simeq R(f^*P_s - \epsilon f^{-1}_* E) \simeq R(f^*P_s)$. Thus $R(P_s- \epsilon E) \simeq R(P_s)$, which is a contradiction to the minimal property of $P_s$. Thus $\mult_{f^{-1}_*E}N_s' = \mult_E N_s$.
\end{proof}

As in Definition \ref{def-pnklt}, we can also define \emph{potentially $s$-klt} or
\emph{potentially $s$-lc pairs}
in a similar manner using the potential $s$-discrepancy $\oa_s(\sigma; X, \Delta)$.
However, we will not go into further details in this paper.

We give an application of the above potentially $s$-singular pairs.

\begin{theorem}\label{thrm-charcy}
Let $X$ be a normal projective variety. Then $X$ is of klt Calabi-Yau type if and only if
$\mathfrak A_s(X)>-1$ and $-K_X$ is effective.
\end{theorem}

\begin{proof}
If $(X, \Delta)$ is a klt Calabi-Yau pair, then $\inf_{\sigma}\oa_s(\sigma; X, \Delta)=\inf_{\sigma}a(\sigma; X, \Delta) > -1$ holds. Thus we obtain $\mathfrak A _s(X)>-1$. It is clear that $-K_X$ is effective.

For the converse, we assume that $\mathfrak A_s(X) > -1$ and $-K_X$ is effective. Fix a sufficiently small $\epsilon >0$ such that $\mathfrak A_s(X) > -1+\epsilon$.
Then there is an effective $\Q$-divisor $\Delta$ such that $-(K_X + \Delta)$ is an effective $\Q$-Cartier divisor and $\oa_s(\sigma; X, \Delta) > -1+\epsilon$ for any divisorial valuation $\sigma$.
Fix a sufficiently large integer $m_0>0$, and consider a log resolution $f : Y \to X$ of $(X, \Delta)$ and $|-m(K_X+\Delta)|$ for a sufficiently divisible positive integer $m < m_0$.
Then we have the decomposition $f^*(-m(K_X+\Delta))=M_m + E_m$ into a base point free divisor $M_m$ and a fixed divisor $E_m$. In Proof of Proposition \ref{prop-pnkltQ}, we proved that
if $-(K_X+\Delta)$ is big, then for any small $\epsilon'>0$, we may take sufficiently divisible and large integers $m'>m$ so that the multiplicity of the effective divisor $\frac{1}{m}E_m - \frac{1}{m'}E_{m'}$ at every prime divisor is bounded above by $\epsilon'$.
Using \cite[Theorem 1.12]{M}, we can show that this statement still holds when $-(K_X+\Delta)$ is effective. Now by the same argument of Proof of Proposition \ref{prop-pnkltQ} using Lemma \ref{prop-sd} (3) instead of Lemma \ref{lem-divzd} (3), we can find an effective divisor $D$ on $X$ such that $-(K_X+\Delta+D) \sim_{\Q} 0$ and $(X, \Delta+D)$ is klt. This finishes our proof.
\end{proof}

\section{Rational connectedness modulo potential non-klt locus}\label{sec-RCC of pklt}
In this section, as another application of the potentially non-klt locus,
we prove Theorem \ref{main thrm2}.
This improves the results on the rational connectedness of uniruled varieties, e.g., \cite[Theorem 0.1]{KMM}, \cite[Corollaire 3.2]{C}, \cite[Theorem 1]{Z}, \cite[Theorem 1.2]{HM2},  \cite[Theorem 1.2]{BP}.
Before proving the theorem, we take a look at some examples and prove some basic properties.

We first prove the connectedness of the potentially non-klt locus.

\begin{proposition}\label{connedted-pnklt}
Let $(X, \Delta)$ be a pair. If $-(K_X + \Delta)$ is big, then $\pnklt(X, \Delta)$ is connected.
\end{proposition}

\begin{proof}
By Proposition \ref{prop-pnkltR}, there exists an effective $\R$-Cartier divisor $D$ such that
$D \sim_{\R} -(K_X+\Delta)$ and $\pnklt(X, \Delta)=\nklt(X, \Delta+D)$.
Since $D$ is big, we have $D \sim_{\R} A + B$ for an ample $\R$-divisor $A$ and an effective $\R$-divisor $B$. Let $\Gamma:=\Delta + (1-\epsilon)D+\epsilon B$ for a sufficiently small rational number $\epsilon >0$ such that $\nklt(X, \Gamma) \subseteq \nklt(X, \Delta+D)$.
Suppose that $\nklt(X, \Gamma) \subsetneq \nklt(X, \Delta+D)$.
For a log resolution $f: Y \to X$ of $(X, \Delta+D)$, we can write
$$
-K_Y = f^*(-(K_X+\Delta+D)) + F = f^*(-(K_X+\Gamma)) -\epsilon f^*D + \epsilon f^*B + F.
$$
Then there is a prime divisor $F_0$ such that $\mult_{F_0} F \geq 1$ and $\mult_{F_0}( -\epsilon f^*D + \epsilon f^*B + F )< 1$.
Since $\epsilon$ is sufficiently small, we can assume that $\mult_{F_0} F = 1$.
Let $f^*D = P+N$ be the divisorial Zariski decomposition. By Lemma \ref{lem-divzd} (2), $f^*B \geq N$. Since
$$
1>\mult_{F_0} (-\epsilon f^*D + \epsilon f^*B + F) = \mult_{F_0} (-\epsilon P + \epsilon (f^*B -N) + F) \geq -\epsilon \mult_{F_0} P + 1,
$$
we get $\mult_{F_0} P >0$. On the other hand, we have
$$
\oa(F_0, X, \Delta)= -\mult_{F_0}(-f^*D + F) - \mult_{F_0} N = \mult_{F_0} P - 1 > -1.
$$
However, $f(F_0) \subseteq \nklt(X, \Delta+D) = \pnklt(X, \Delta)$, so we get a contradiction. Thus $\nklt(X, \Gamma) = \nklt(X, \Delta+D)$.
Since $-(K_X + \Gamma) \sim_{\R} \epsilon A$ is ample, the connectedness of $\pnklt(X,\Delta)=\nklt(X, \Gamma)$ follows from the connectedness lemma of Shokurov and Koll\'{a}r on non-klt locus (\cite[Theorem 17.4]{Ko+}).
\end{proof}

We can easily construct a variety $X$ such that $-K_X$ is only pseudoeffective and $\pnklt(X, 0)$ is not connected.

\begin{example}
Let $S$ be an extremal rational elliptic surface with two singular fibers $F_1, F_2$ of the same type $\widetilde{D}_4$ (see e.g., \cite[Example 5.1]{HP1}). Each singular fiber $F_i$ has an irreducible component $G_i$ such that $\mult_{G_i} F_i=2$ for $i=1,2$.
Let $f: \widetilde{S} \to S$ be a blow-up at two points $p_1, p_2$, where each $p_i$ is a general point on $G_i$ for $i=1,2$. We can easily check that $\kappa(-K_{\widetilde{S}})=0$ and $\pnklt(\widetilde{S},0)=f^{-1}_* G_1 \cup f^{-1}_* G_2$. In particular, $\pnklt(\widetilde{S},0)$ is not connected.
\end{example}

Recall that by \cite[Theorem 1.2]{BP} if $(X,\Delta)$ is a pair with an effective $\Q$-divisor $\Delta$ on $X$ such that $-(K_X+\Delta)$ is big, then $X$ is rationally connected modulo $\nklt(X,\Delta) \cup \nnef(-(K_X+\Delta))$.
(In \cite{BP}, the notation $\nnef$ is used to mean $\B_-$, the restricted base locus.
However, $\nklt(X,\Delta) \cup \B_-(-(K_X+\Delta))=\nklt(X,\Delta) \cup \nnef(-(K_X+\Delta))$
because $\B_-(-(K_X+\Delta)$ and $\nnef(-(K_X+\Delta))$ can differ only in $\nklt(X,\Delta)$.
Our current usage of the notation $\nnef$ seems more common.)
By \cite[Theorem 1.2]{HM2}, if $f: X \to Y$ is a birational morphism such that $-(K_Y+\Delta)$ is big and semiample, then $X$ is rationally chain connected modulo $f^{-1}(\nklt(Y,\Delta))$. We remark that this locus $f^{-1}(\nklt(Y,\Delta))$ contains $\nklt(X, f^{-1}_*\Delta)\cup \nnef(-(K_X+f^{-1}_*\Delta))$.

The following example shows that there exists a variety $X$ having a divisor $\Delta$ such that $-(K_X+\Delta)$ is big
which is not rationally connected modulo any irreducible component of
$\nklt(X,\Delta)\cup\nnef(-(K_X+\Delta))$ not contained in $\pnklt(X,\Delta)$.
Thus Theorem \ref{main thrm2} improves the main result of \cite{BP}.

\begin{example}\label{improvebp}
Let $C$ be a smooth projective curve of genus $g > 1$ and $A$ a divisor of degree $e > 2g-2$ on $C$.
Consider the ruled surface $X:=\P(\mathcal{O}_C \oplus \mathcal{O}_C(-A))$ with the ruling $f : X \to C$.
Then $-K_X$ is big, and the negative part of the Zariski decomposition
$-K_X = P+N$ is given by $N=\left(1+\frac{2g-2}{e} \right)C_0$
where $C_0$ is a section of $f$ corresponding to the canonical projection
$\mathcal{O}_C(-A) \oplus \mathcal{O}_C \to \mathcal{O}_C(-A)$.
Let $\pi : \widetilde{X} \to X$ be the blow-up of $X$ at a point on $C_0$ with the exceptional divisor $E$.
The Zariski decomposition $-K_{\widetilde{X}}=P' + N'$ is given by $P'=\pi^*P$ and $N'=\pi^*N - E= \left(1+\frac{2g-2}{e} \right)\pi_*^{-1}C_0 + \frac{2g-2}{e}E$.
Note that $0 < \mult_E(N') < 1$.
Thus $\nklt(\widetilde{X},0)\cup \nnef(-K_{\widetilde{X}})$ consists of two components $E$ and $\pi_*^{-1}C_0$, and $\pnklt(\widetilde{X},0)$ consists of the single component $\pi_*^{-1}C_0$.
Clearly, $\widetilde{X}$ is not rationally connected modulo $E$, but $\widetilde{X}$ is rationally connected modulo $\pi_*^{-1}C_0$.
\end{example}

We give an example of the second case (2) of Theorem \ref{main thrm2}.

\begin{example}\label{secondcaseoccurs}
Let $X=\P^n \times A$ where $A$ is an abelian variety.
Then $-K_X$ is semiample so that $\pnklt(X, 0)=\emptyset$. However, $X$ is not rationally connected. We can take a non-zero effective divisor $\Delta \sim_{\Q} -K_X$ such that $(X, \Delta)$ is a klt Calabi-Yau pair. We can check that $X$ is rationally connected modulo every irreducible component of $\Delta$.
\end{example}

Now we prepare some ingredients for the proof of Theorem \ref{main thrm2}.

\begin{lemma}[cf. {\cite[Lemma 4.2]{HM2}}]\label{dominate-rcc}
For a normal projective variety $X$, let $f : Y \to X$ be a resolution, and let $\pi : Y \dashrightarrow Z$ be the MRC-fibration. If a subset $V$ of $Y$ dominates $Z$ via $\pi$, then $X$ is rationally connected modulo $f(V)$.
\end{lemma}

\begin{proof}
If $V = \emptyset$, then $Z$ must be a point so that $X$ is rationally connected. Assume that $V \neq \emptyset$.
We need to show that for a general point $x \in X$, there is a rational curve $R$ passing through $x$ and intersecting $f(V)$.
Let $y \in Y$ be a general point with $f(y)=x$. Since a general fiber of $\pi$ is a smooth rationally connected variety, there is a rational curve $R'$ passing through $y$ and intersecting $V$. Then the rational curve $R=f(R')$ passes through $x$ and intersects $f(V)$.
\end{proof}

The following lemma is essentially due to Campana (\cite[Theorem 4.13]{Ca}).

\begin{lemma}[{\cite[Corollary 2.2]{BP}}]\label{cam-weakpos}
Let $\pi : Y \to Z$ be a contraction of smooth projective varieties and $\Gamma$ an effective $\Q$-divisor on $Y$ such that $(Y, \Gamma)$ is lc on a general fiber of $\pi$. Let $W$ be a general fiber of $\pi$ and suppose that $\kappa(W, m(K_{Y}+\Gamma)|_W ) \geq 0$. Then for every ample $\Q$-divisor $H$ on $Z$, we have $\kappa(K_{Y/Z}+\Gamma + \pi^*H) > 0$.
\end{lemma}

We now prove Theorem \ref{main thrm2} (1).
The proof is similar to that of \cite[Lemma 4.2]{BP}.

\begin{theorem}\label{thrm-RC via pnklt}
Let $(X, \Delta)$ be a pair such that $-(K_X+\Delta)$ is big. Then $X$ is rationally connected modulo  $\pnklt(X, \Delta)$.
\end{theorem}

\begin{proof}
By Corollary \ref{cor-r to q}, we may assume that $\Delta$ is a $\Q$-divisor and $-(K_X+\Delta)$ is big $\Q$-Cartier. Using Proposition \ref{prop-pnkltQ}, we can take an effective $\Q$-Cartier divisor $D$ such that $D \sim_{\Q} -(K_X+\Delta)$ and $\nklt(X, \Delta+D)=\pnklt(X, \Delta)$.
Since $\Delta+D$ is big, we may write $\Delta+D \sim_{\Q} A+B$ for an ample $\Q$-divisor $A$ and an effective $\Q$-divisor $B$. Let $\Delta' := (1-\epsilon)(\Delta+D) + \epsilon A + \epsilon B$ for a sufficiently small rational number $\epsilon >0$. Then $K_X + \Delta' \sim_{\Q} 0$ and $\nklt(X, \Delta') \subseteq \nklt(X, \Delta+D)$.
For a log resolution $f: Y \to X$ of $(X, \Delta')$, we can write
$$
K_{Y} + \Gamma = f^*(K_X + \Delta') + E
$$
where $\Gamma$ and $E$ are effective divisors having no common components.
Note that since $E$ is $f$-exceptional, we have $0=\kappa(K_X+\Delta')=\kappa(K_Y+\Gamma)$.
Note also that $\nklt(X, \Delta') =f (\nklt(Y,\Gamma))$.
Now consider the MRC-fibration $\pi : Y \dashrightarrow Z$.
If $Z$ is a point, then $Y$ is rationally connected and so is $X$.
Let us assume that $\dim Z >0$.
Since $Z$ is not uniruled, it follows from \cite[Theorem 3.5]{BP} that $\nklt(Y, \Gamma)$ dominates $Z$.
Thus there exists an irreducible component $V'$ of $\pnklt(Y,\Gamma)$ which dominates $Z$.
Now Lemma \ref{dominate-rcc} implies that $X$ is rationally connected modulo $V=f(V')\subseteq\nklt(X, \Delta')$.
\end{proof}

\begin{corollary}
Let $(X,\Delta)$ be a potentially klt pair and $-K_X$ is big. Then $X$ is rationally connected.
\end{corollary}

\begin{proof}
Since $(X, \Delta)$ is a klt pair, we can take a small $\Q$-factorialization $q: X' \to X$. Note that $(X', q^{-1}_*\Delta)$ is a potentially klt pair and $-K_{X'}$ is big.
By Theorem \ref{thrm-charft}, we may pick an effective $\R$-divisor $\Gamma$ such that $-(K_{X'} + \Gamma)$ is big $\R$-Cartier and $\pnklt(X', \Gamma) = \emptyset$. By Theorem \ref{thrm-RC via pnklt}, $X'$ is rationally connected, so is $X$.
\end{proof}

It also follows from \cite[Theorem 1]{Z} or \cite[Theorem 1.2]{HM2} since $X$ is of Fano type by Theorem \ref{thrm-charft}.

\begin{corollary}\label{rccpnklt}
Let $(X, \Delta)$ be a pair such that $-(K_X+\Delta)$ is big.
If $\pnklt(X, \Delta)$ is rationally chain connected, then $X$ is also rationally chain connected.
\end{corollary}

\begin{proof}
Immediate from Theorem \ref{thrm-RC via pnklt}.
\end{proof}

In particular, when $-K_X$ is $\Q$-Cartier and big, the rational chain connectedness of $\pnklt(X,0)$ implies the same property of $X$. In the surface case, the converse also holds.

\begin{proposition}\label{prop-rcc of pnklt&X}
Let $X$ be a normal projective $\Q$-Gorenstein surface such that $-K_X$ is big. Then $X$ is rationally chain connected if and only if $\pnklt(X, 0)$ is rationally chain connected.
\end{proposition}

\begin{proof}
The `if' direction is obvious by Corollary \ref{rccpnklt}.
To prove the `only if' direction, we assume that $X$ is rationally chain connected.
Let $g: Z \to X$ be the minimal resolution. Then the anticanonical morphism $h : Z \to Y$ factors through $X$, and the morphism $f: X \to Y$ with $h = f \circ g$ is also the anticanonical morphism.
We can assume that $Z \neq \P^2$. Then there is the canonical ruling $\pi : Z \to C$ to a smooth projective curve so that $Z$ is rationally chain connected modulo a section $S$ of $\pi$.
Note that $g(\pnklt(Z, 0))=\pnklt(X, 0)$. If $\pnklt(Z, 0)$ is rationally chain connected, then so is $\pnklt(X, 0)$. Now suppose that $\pnklt(Z, 0)$ is not rationally chain connected. Then $S$ is the only non-rational component of $\pnklt(Z, 0)$, but it is contracted by $h$.
Since $X$ is rationally chain connected, $S$ is also contracted by $g$.
Thus $\pnklt(X, 0)$ contains only rational curves.
By Proposition \ref{connedted-pnklt}, $\pnklt(X, 0)$ is rationally chain connected.
\end{proof}

Unfortunately, Proposition \ref{prop-rcc of pnklt&X} does not hold in higher dimensions by the following example.

\begin{example}\label{nonft}
We construct a smooth rational variety $X$ such that $-K_X$ is big and $\pnklt(X, 0)$ is a smooth elliptic curve. Consider a $2$-dimensional linear subspace $P$ in $\P^n$ for any integer $n \geq 3$, and take a smooth plane cubic curve $C \subseteq P \simeq \P^2$. Let $L$ be a hyperplane divisor of $\P^n$ and $\pi : X \rightarrow \P^n$ the blow-up at twelve general points $p_1, \ldots, p_{12}$ on $C$ in such a way that the line bundle $\mathcal{O}_C(p_1 + \cdots + p_{12} - 4L)$ is a non-torsion class in $\Pic^0(C)$ with exceptional divisors $E_1, \ldots, E_{12}$. Then we have
$$
-K_X = (n+1)\pi^{*}L - (n-1)E
$$
where $E:=E_1 + \cdots + E_{12}$.
First, we note that $4H-E$ is nef and big, but it is not semiample (see \cite[p.158]{posI}). Thus $X$ is not a Mori dream space, and in particular, $X$ is not of Fano type.
Now we claim that $-K_X$ is big and $\pnklt(X, 0)=C$ where we use the same notation for the strict transform of $C$.
To show the claim, consider a hyperplane $H$ and a cubic hypersurface $V$ in $\P^n$ containing $C$.
Since $-K_X \sim \pi^*2H + (n-1)(\pi^*H - E)$, it follows that $-K_X$ is big.
On the other hand, we have
$$
-K_X \sim (\pi^*V - E) + (n-2)(\pi^*H - E) \geq 0.
$$
By varying such hyperplanes and cubic hypersurfaces, we see that $\B(-K_X) \subseteq C$ so that $\nnef (-K_X) \subseteq C$. If $\nnef(-K_X) \neq C$, then $\nnef(-K_X)=\emptyset$ so that $\pnklt(X, 0)=\emptyset$. By Theorem \ref{thrm-charft}, $X$ is of Fano type, which is a contradiction. Thus $\nnef(-K_X)=C$. Now we can easily see that $\pnklt(X, 0)=C$.
\end{example}

Moreover, when $-K_X$ is only pseudoeffective, Proposition \ref{prop-rcc of pnklt&X} does not hold even in the surface case. For instance, if $f: \widetilde{S} \to S$ is a blow-up at a general point on a general fiber $C$ of a rational elliptic surface $S$, then $\pnklt(\widetilde{S}, 0)=f^{-1}_*C$ is not rational.

Finally, we prove Theorem \ref{main thrm2} (2). The proof is similar to \cite[Proof of Main Theorem]{Z2} and \cite[Proof of Theorem 1.6]{BP}.

\begin{theorem}\label{thrm-RC pseff}
Let $(X, \Delta)$ be a pair such that $-(K_X + \Delta)$ is pseudoeffective. Then either $X$ is rationally connected modulo $\pnklt(X, \Delta)$ or $\kappa(Z)=0$ where $\pi: X \dashrightarrow Z$ is the MRC-fibration of $X$.
For the latter case, $\pi$ is semistable in codimension two, and $X$ is rationally connected modulo every irreducible component of $\Delta$ if $\Delta \neq 0$.
\end{theorem}

\begin{proof}
Let $f : Y \to X$ be a log resolution of $(X, \Delta)$ and $\pi : Y \dashrightarrow Z$ the MRC-fibration.
Possibly by taking further blow-ups, we may assume that $Z$ is smooth and $\pi : Y \to Z$ is a morphism.
We can easily take a sequence $\{ A_i \}$ of ample $\R$-divisors such that $-(K_X+\Delta)+A_i$ is big $\Q$-Cartier for every $i$ and $\lim_{i \to \infty} A_i = 0$.
For each $i$, by Proposition \ref{pnklt-locus-pseff}, there is an effective $\Q$-Cartier divisor $D_i$ such that $D_i \sim_{\Q} -(K_X + \Delta) + A_i$ and $\nklt(X, \Delta+D_i) \subseteq \pnklt(X, \Delta)$.
We now write
$$
K_Y + \Gamma_i = f^*(K_X+\Delta+D_i)+E_i
$$
where $\Gamma_i$ and $E_i$ are effective divisors having no common components.
Note that $\Gamma_i \geq f^{-1}_*\Delta$ and $E_i$ is $f$-exceptional.
For an ample $\Q$-divisor $H$ on $Z$, we can take an ample $\Q$-divisor $A$ on $X$ such that $f^*A - \pi^*H$ is again an ample $\Q$-divisor on $Y$. Choose an effective ample $\Q$-divisor $A'$ on $Y$ such that $A' \sim_{\Q} f^*A - \pi^*H$ and $\nklt(Y, \Gamma_i) = \nklt(Y, \Gamma_i +\frac{1}{i}A')$ for all $i$.
Let $\Gamma_i':=\Gamma_i + \frac{1}{i}A'$ and $A_i':=A_i + \frac{1}{i}A$.
We still have $\lim_{i \to \infty} A_i' = 0$.

Note that for each $i$,
$$
f(\nklt(Y, \Gamma_i'))=f(\nklt(Y, \Gamma_i))=\nklt(X, \Delta+D_i)\subseteq\pnklt(X, \Delta).
$$
If $\nklt(Y, \Gamma_i')$ dominates $Z$ or $Z$ is a point,
then $X$ is rationally connected modulo $\pnklt(X, \Delta)$ by Lemma \ref{dominate-rcc}.
Thus we assume that $\nklt(Y, \Gamma_i')$ does not dominate $Z$ and $\dim Z >0$.
We first note that $(Y, \Gamma_i')$ is klt on a general fiber of $\pi$.
Furthermore, by \cite[Corollary 1.4]{GHS}, $Z$ is not uniruled.

Now we prove that $\kappa(Z) =0$ and
if $\Delta \neq 0$, then every component of $f^{-1}_*\Delta$ dominates $Z$.
Using the semistable reduction theorem, covering trick, and the flattening of $\pi$, we can find a finite morphism $p: \widetilde{Z} \to Z$ and the induced morphism $\widetilde{\pi} : \widetilde{Y} \to \widetilde{Z}$ from a resolution $\widetilde{Y} \to Y \times_Z \widetilde{Z}$ such that it is semistable in codimension two. Let $q: \widetilde{Y} \to Y$ be the induced morphism (see \cite[Proof of Proposition 1]{Z}).
$$
\xymatrix{
X& Y \ar[l]_{f} \ar[d]_{\pi}& \widetilde{Y} \ar[l]_q \ar[d]^{\widetilde{\pi}} \\
&Z&\widetilde{Z} \ar[l]_p .
}
$$
We have
$$
K_{Y/Z} + \Gamma_i' \sim_{\Q} f^*A_i' +E_i - \pi^*(K_Z + \frac{1}{i}H).
$$
Then by \cite[Proof of Proposition 1]{Z}, we obtain
$$
K_{\widetilde{Y}/\widetilde{Z}} + \widetilde{\Gamma}_i' \sim_{\Q} \widetilde{E}_i - \widetilde{\pi}^*p^*(K_Z +\frac{1}{i} H) + q^*f^*A_i' - V_i
$$
where
\begin{enumerate}
\item $\widetilde{E}_i$ is $f \circ q$-exceptional,
\item $V_i$ is an effective divisor which is $\widetilde{\pi}$-vertical,
\item If an irreducible component $F$ of $\widetilde{E}_i$ is $\widetilde{\pi}$-horizontal, then $\mult_{F} \widetilde{E}_i \geq 0$,
\item $(\widetilde{Y}, \widetilde{\Gamma}_i')$ is klt on a general fiber of $\widetilde{\pi}$,
\item $p^*H = \widetilde{H}$ is ample, and
\item If a divisor $G$ on $\widetilde{Y}$ such that the codimension $\widetilde{\pi}(G)$ is at least 2, then $G$ is $f \circ q$-exceptional.
\end{enumerate}
Since $(\widetilde{Y}, \widetilde{\Gamma}_i' + V_i)$ is klt on a general fiber of $\widetilde{\pi}$, we can assume that $V_i$ contains only $\widetilde{\pi}$-vertical irreducible components of $(f \circ q)_*^{-1}(\Delta)$. Thus $V_i$ does not depend on $i$, so we put $V:=V_i$.
We rewrite
$$
K_{\widetilde{Y}/\widetilde{Z}} + \widetilde{\Gamma}_i' \sim_{\Q} \widetilde{E}_i - \widetilde{\pi}^*p^*(K_Z +\frac{1}{i} H) + q^*f^*A_i' - V.
$$

If we denote by $W$ a general fiber of $\widetilde{\pi}$,
we obtain $(K_{\widetilde{Y}} + \widetilde{\Gamma}_i + V)|_W \sim_{\Q} \widetilde{E}_i + q^*f^* A_i'$ which is effective.
By Lemma \ref{cam-weakpos}, we have $h^0(K_{\widetilde{Y}/\widetilde{Z}}+\widetilde{\Gamma}_i' + V + \frac{1}{i}\widetilde{\pi}^*\widetilde{H}) > 0$, thus $\widetilde{E} - \widetilde{\pi}^* p^* K_Z + q^* f^* A_i' - V$ is effective. We can choose a family of general complete intersection curves $\widetilde{C}$ on $\widetilde{Y}$ such that $\widetilde{C}$ does not intersect with the exceptional locus of $f \circ q : \widetilde{Y} \to X$. Then we get
$$
\widetilde{C} \cdot \widetilde{\pi}^* p^* K_Z \leq \widetilde{C} \cdot (q^*f^* A_i' - V).
$$
If $C \cdot \pi'^* p^* K_Z <0$, then by \cite[Corollary 0.3]{BDPP}, $Z$ is uniruled, which is a contradiction.
Thus $C' \cdot \pi'^* p^* K_Z =0$.
There exists a covering family of curves $C$ on $Y$ such that $C \cdot \pi^* K_Z \leq C \cdot (A_i' - q_* V)$.
As $i \to \infty$ (so that $A_i' \to 0$), we see that $C \cdot \pi^*K_Z=0$, $(f \circ q)_* V=0$, and the MRC-quotient $X \dashrightarrow Z$ is semistable in codimension two.
If  $\Delta \neq 0$, then since $V$ cannot contain any irreducible component of $(f\circ q)_*^{-1}(\Delta)$, every irreducible component of $\Delta$ dominates $Z$. Now by considering the divisorial Zariski decomposition of $K_Z$ and by applying \cite[Theorem 9.8]{BDPP} and \cite[Corollary V.4.9]{Na} as in \cite[Proof of Main Theorem]{Z2} or \cite[Proof of Theorem 1.6]{BP}, we conclude that $\kappa(Z)=0$.
\end{proof}

\begin{remark}
In Theorem \ref{thrm-RC pseff}, we do not assume that $X$ is uniruled. However, if $\Delta \neq 0$ and $-(K_X+\Delta)$ is pseudoeffective, then $X$ is uniruled by \cite[Corollary 0.3]{BDPP}.
\end{remark}

\section*{Acknowledgements}
We would like to thank professors Vyacheslav V. Shokurov for valuable suggestions, Florin Ambro for interesting discussions, and Yoshinori Gongyo for useful comments.
We wish to express our deep gratitude to Dr. Zhengyu Hu for kindly pointing out a mistake in previous version of the proof of Proposition \ref{prop-pnkltQ}.
S. Choi was supported by IBS-R003-D1.


\end{document}